\documentclass{article}
\usepackage{amsfonts}
\usepackage{amsmath}
\usepackage{graphicx}
\newtheorem{theorem}{Theorem}

\newtheorem{example}[theorem]{Example}

\newenvironment{proof}[1][Proof]{\noindent\textbf{#1.} }{\ \rule{0.5em}{0.5em}}

\begin{document}

\title{Weingarten map of the hypersurface in 4-dimensional Euclidean space and its applications}
\date{\vspace{-5ex}}

\author{Salim Y\"{U}CE\\\\
Yildiz Technical University,\\
Faculty of Arts and Sciences,\,\
Department of Mathematics,\\
Davutpa{\c{s}}a Campus, 34220,
Esenler, Istanbul,  TURKEY\\
E-mail: sayuce@yildiz.edu.tr
}
\maketitle

\renewcommand{\thefootnote}{}


\renewcommand{\thefootnote}{\arabic{footnote}}
\setcounter{footnote}{0}

\begin{abstract}
In this paper, by taking into account the beginning of the hypersurface theory in Euclidean space $E^4$, a practical method for the matrix of the Weingarten map (or the shape operator) of an oriented hypersurface $M^3$ in $E^4$ is obtained. By taking this efficient method, it is possible to study of the hypersurface theory in $E^4$ which is analog the surface theory in $E^3$. Furthermore, the Gaussian curvature, mean curvature, fundamental forms and Dupin indicatrix of $M^3$ is introduced.
\end{abstract}
\section{Introduction} 

\addcontentsline{toc}{section}{ }

Let $ x = \sum\limits_{i = 1}^4 {{x_i}{e_i}} ,\,\,y = \sum\limits_{i = 1}^4 {{y_i}{e_i}} ,\,\,z = \sum\limits_{i = 1}^4 {{z_i}{e_i}} $ be three vectors in $\mathbb{R}^4$, equipped with the standard inner product given by

\[\left\langle {x,y} \right\rangle  = {x_1}{y_1} + {x_2}{y_2} + {x_3}{y_3} + {x_4}{y_4},\]
where $ \{ {e_1}, {e_2}, {e_3}, {e_4} \}  $ is the standard basis of $\mathbb{R}^4.$
The norm of a vector $ x\in \mathbb{R}^4 $ is given by $ \left\| x \right\| = \sqrt {\left\langle {x,x} \right\rangle }.$ The vector product (or the ternary product or cross product) of the vectors $x,y,z \in \mathbb{R}^4$ is defined by
\begin{equation} \label{1}
	x \otimes y \otimes z = \left| {\begin{array}{*{20}{c}}
			{{e_1}}&{{e_2}}&{{e_3}}&{{e_4}}\\
			{{x_1}}&{{x_2}}&{{x_3}}&{{x_4}}\\
			{{y_1}}&{{y_2}}&{{y_3}}&{{y_4}}\\
			{{z_1}}&{{z_2}}&{{z_3}}&{{z_4}}
	\end{array}} \right|.
\end{equation}
Some properties of the vector product are given as follows: (for the vector product in  $\mathbb{R}^4 $, see \cite{nam1, nam3, nam5}

	\begin{itemize}
	\item[i.]
	$\left\{ {\begin{array}{*{20}{l}}
		{{e_1} \otimes {e_2} \otimes {e_3} =  - {e_4}}\\
		{{e_2} \otimes {e_3} \otimes {e_4} = {\mkern 1mu} {\mkern 1mu} {\mkern 1mu} {\mkern 1mu} {\mkern 1mu} {e_1}}\\
		{{e_3} \otimes {e_4} \otimes {e_1} =  - {e_2}}\\
		{{e_4} \otimes {e_1} \otimes {e_2} = {\mkern 1mu} {\mkern 1mu} {\mkern 1mu} {\mkern 1mu} {\mkern 1mu} {e_3}}\\
		{{e_3} \otimes {e_2} \otimes {e_1} = {\mkern 1mu} {\mkern 1mu} {\mkern 1mu} {\mkern 1mu} {\mkern 1mu} {e_4}}
		\end{array}} \right.$\\
		
	\item[ii.] \begin{equation} \label{2}
	{\left\| {x \otimes y \otimes z} \right\|^2} = \left| {\begin{array}{*{20}{c}}
			{\left\langle {x,x} \right\rangle }&{\left\langle {x,y} \right\rangle }&{\left\langle {x,z} \right\rangle }\\
			{\left\langle {y,x} \right\rangle }&{\left\langle {y,y} \right\rangle }&{\left\langle {y,z} \right\rangle }\\
			{\left\langle {z,x} \right\rangle }&{\left\langle {z,y} \right\rangle }&{\left\langle {z,z} \right\rangle }
			\end{array}} \right|\end{equation}
		
	\item[iii.] $\left\langle {x \otimes y \otimes z,t} \right\rangle  = \det \left( {x,y,z,t} \right)$.
\end{itemize}

Let \, $ \,M^3\,$  \, be an oriented\, 3-\,dimensional hypersurface in\,\, 4-dimensional Euclidean space $E^4$. Let examine the implicit and parametric equations of $ M^3$. Firstly; the implicit equation of  $ M^3$ can be defined by\\
\begin{equation} \label{3}
	M^3=\left \{X \in E^4 | f:U\subset E^4\overset{diff.}{\rightarrow} \mathbb{R},\, f(X)=const. \,\,\, {\overrightarrow {\nabla f} \left| {_P}\neq 0, \right. P\in M^3}  \right \}
\end{equation}
where $ {\overrightarrow {\nabla f} \left| {_P} \right.}$ is the gradient vector of $M^3.$ The unit normal vector field of $M^3$ is defined by $N = \frac{{\overrightarrow {\nabla f} }}{{\left\| {\overrightarrow {\nabla f} } \right\|}}$.

The Weingarten map (or the shape operator) of  $ M^3$ is defined by
\begin{equation*}
S:\chi \left( M^3 \right) \to \chi \left( M^3 \right),\,\,S\left( X \right) = {D_X}N,
\end{equation*}
where $D$ is the connection of $ E^4$ and $\chi \left( M^3 \right)$ is the space of vector fields of $M^3$. Then the Gauss curvature $K$ and mean curvature $ H $ of $M^3$ are given by  $K=detS$ and $ H=\frac{1}{3} Tr S $, respectively. Also, the $q-th$ fundamental forms of $ M^3$ are given by \cite{nam6},
\[{I^q}\left( {X,Y} \right) = \left\langle {{S^{q - 1}}\left( X \right),Y} \right\rangle ,\,\,{\rm{\forall }}\,\,X,Y \in \chi \left( {{M^3}} \right).\]

Secondly, to examine parametric form of the hypersurface $M^3$ given by the implicit equation in the eq \eqref{3}, let consider
\begin{equation*}
\begin{array}{l}
\phi :U \subset {\mathbb{R}^3} \to {E^4}\\
\,\,\,\,\,\,\,\,\left( {u,v,w} \right) \to \phi \left( {u,v,w} \right)=\left( {{\varphi _1}\left( {u,v,w} \right),{\varphi _2}\left( {u,v,w} \right),{\varphi _3}\left( {u,v,w} \right),{\varphi _4}\left( {u,v,w} \right)} \right)
\end{array}
\end{equation*}
where $ \left( {u,v,w} \right) \in R \subset {\mathbb{R}^3} $ and ${\varphi _i}, 1\le i \le 4 $ are the real functions defined on $R$.\\
${M^3} = \phi \left( R \right) \subset {E^4}$ is a hypersurface if only if the frame field $\left\{ {{\phi _u},{\phi _v},{\phi _w}} \right\}$ of $M^3$ is linearly independent system. It can be also seen by taking the Jacobian matrix ${\left[ \phi  \right]_*} = \left[ {\begin{array}{*{20}{c}}
	{{\phi _u}}&{{\phi _v}}&{{\phi _w}}
	\end{array}} \right]$ of the differential map of $\phi$. It is clear that if rank ${\left[ \phi  \right]_*} =3$, then the vector system $\left\{ {{\phi _u},{\phi _v},{\phi _w}} \right\}$ is linearly independent. Furthermore, $ {{\phi _u},{\phi _v},{\phi _w}} $ are the tangent vectors of the parameter curves $\alpha \left( u \right) = \phi \left( {u,{v_0},{w_0}} \right) $, $  \beta \left( v \right) = \phi \left( {{u_0},v,{w_0}} \right) $ and $ \gamma \left( w \right) = \phi \left( {{u_0},{v_0},w} \right) $, respectively. Then the unit normal vector field of $M^3$ is defined by
\begin{equation} \label{4}
	N = \frac{{{\phi _u} \otimes {\phi _v} \otimes {\phi _w}}}{{\left\| {{\phi _u} \otimes {\phi _v} \otimes {\phi _w}} \right\|}}
\end{equation}
and it has the following properties:
\begin{equation} \label{5}
	\left\langle {N,{\phi _u}} \right\rangle  = \left\langle {N,{\phi _v}} \right\rangle  = \left\langle {N,{\phi _w}} \right\rangle  = 0.
\end{equation}
By using the Weingarten operator the below equalities can be written
\begin{equation} \label{9}
\begin{array}{l}
   S\left( {{\phi _u}} \right) = {D_{{\phi _u}}}N = \frac{{\partial N}}{{\partial u}} \\
   \\
   S\left( {{\phi _v}} \right) = {D_{{\phi _v}}}N = \frac{{\partial N}}{{\partial v}}\\
   \\
   S\left( {{\phi _w}} \right) = {D_{{\phi _w}}}N = \frac{{\partial N}}{{\partial w}}.
\end{array}
\end{equation}

\section{The matrix of the Weingarten map of hypersurface $M^3$ in $E^4$}
\addcontentsline{toc}{section}{ }

In this original section, a practical method for the matrix of the Weingarten map of hypersurface $M^3$ in $E^4$ is introduced.

Let $M^3$ be an oriented hypersurface with the parametric equation  $\phi \left( {u,v,w} \right)$.   Then $ \{{\phi _u},{\phi _v},{\phi _w}\} $ is linearly independent and we also can write
\begin{equation} \label{6}
	\begin{array}{l}
		S\left( {{\phi _u}} \right) = a_{11}\,{\phi _u} + a_{21}\,{\phi _v} + a_{31}\,{\phi _w}\\
		S\left( {{\phi _v}} \right) = a_{12}\,{\phi _u} + a_{22}\,{\phi _v} + a_{32}\,{\phi _w}\\
		S\left( {{\phi _w}} \right) = a_{13}\,{\phi _u} + a_{23}\,{\phi _v} + a_{33}\,{\phi _w}
	\end{array}
\end{equation}
and the Weingarten matrix is given by
\[S = \left( {\begin{array}{*{20}{c}}
	a_{11}&a_{12}&a_{13}\\
	a_{21}&a_{22}& a_{23}\\
	a_{31}& a_{32}&a_{33}
	\end{array}} \right),\]
where $a_{ij}\in \mathbb{R}, 1 \le i,j \le 3.$
Using the equation \eqref{6}, we have the following systems of linear equations:

\begin{equation} \label{7}
	\begin{array}{l}
		\left\{ \begin{array}{l}
			\left\langle {S\left( {{\phi _u}} \right),{\phi _u}} \right\rangle  = a_{11} \phi _{11} +a_{21} \phi _{12} + a_{31}\phi _{13}\\
			\left\langle {S\left( {{\phi _u}} \right),{\phi _v}} \right\rangle  = a_{11} \phi _{12} + a_{21} \phi _{22} + a_{31} \phi _{23}\\
			\left\langle {S\left( {{\phi _u}} \right),{\phi _w}} \right\rangle  = a_{11} \phi _{13} + a_{21} \phi _{23} + a_{31} \phi _{33},
		\end{array} \right.\\
		\\
		\left\{ \begin{array}{l}
			\left\langle {S\left( {{\phi _v}} \right),{\phi _u}} \right\rangle  = a_{12} \phi _{11} +a_{22} \phi _{12} + a_{32} \phi _{13}\\
			\left\langle {S\left( {{\phi _v}} \right),{\phi _v}} \right\rangle  =a_{12} \phi _{12} + a_{22} \phi _{22} + a_{32}  \phi _{23}\\
			\left\langle {S\left( {{\phi _v}} \right),{\phi _w}} \right\rangle  = a_{12} \phi _{13} + a_{22} \phi _{23} +a_{32}  \phi _{33},
		\end{array} \right.\\
		\\
		\left\{ \begin{array}{l}
			\left\langle {S\left( {{\phi _w}} \right),{\phi _u}} \right\rangle  = a_{13}  \phi _{11} + a_{23} \phi _{12} + a_{33}  \phi _{13}\\
			\left\langle {S\left( {{\phi _w}} \right),{\phi _v}} \right\rangle  = a_{13} \phi _{12} + a_{23}  \phi _{22} + a_{33} \phi _{23}\\
			\left\langle {S\left( {{\phi _w}} \right),{\phi _w}} \right\rangle  = a_{13} \phi _{13} +a_{23}  \phi _{23} + a_{33} \phi _{33},
		\end{array} \right.
	\end{array}
\end{equation}
where
\begin{equation} \label{8}
	\begin{array}{l}
		\left\langle {{\phi _u},{\phi _u}} \right\rangle  = \phi _{11},\,\,\left\langle {{\phi _u},{\phi _v}} \right\rangle  = \phi _{12},\,\,\left\langle {{\phi _u},{\phi _w}} \right\rangle  = \phi _{13},\\
		\left\langle {{\phi _v},{\phi _v}} \right\rangle  = \phi _{22},\,\, \left\langle {{\phi _v},{\phi _w}} \right\rangle  = \phi _{23},\,\,\left\langle {{\phi _w},{\phi _w}} \right\rangle  = \phi _{33}.
	\end{array}
\end{equation}
Since the system $ \{{\phi _u},{\phi _v},{\phi _w}\} $ is linearly independent, using the equations \eqref{2} and \eqref{8}, we have
\[{\left\| {{\phi _u} \otimes {\phi _v} \otimes {\phi _w}} \right\|^2} = \left| {\begin{array}{*{20}{c}}
	\phi _{11}&\phi _{12}&\phi _{13}\\
	\phi _{12}&\phi _{22}&\phi _{23}\\
	\phi _{13}&\phi _{23}&\phi _{33}
	\end{array}} \right| \ne 0.\]
Also, 3-linear equation systems given by the equation \ref{7} have the determinant
\[\left| {\begin{array}{*{20}{c}}
	\phi _{11}&\phi _{12}&\phi _{13}\\
	\phi _{12}&\phi _{22}&\phi _{23}\\
	\phi _{13}&\phi _{23}&\phi _{33}
	\end{array}} \right| = \Delta .\]
Because of the property  ${\left\| {{\phi _u} \otimes {\phi _v} \otimes {\phi _w}} \right\|^2} = \Delta  \ne 0$, these 3-linear equations systems can be solved by Cramer method. Then using the equations \eqref{9}, \eqref{7} and \eqref{8} the matrix $S$ of the Weingarten map in $M^3$ can be found. Although $S$ is a symmetric linear operator, the matrix presentation  $(a_{ij})$ of $S$ with respect to $\{{\phi _u},{\phi _v},{\phi _w}\}$ is not necessary to be symmetric because the system $\{ {\phi _u},{\phi _v}, {\phi _w} \}$ is not orthonormal.
\subsection{Special Case}
If we take the orthogonal frame field $ \{{\phi _u},{\phi _v},{\phi _w}\} $ of the hypersurface $M^3,$ then we have $\phi _{12}=\phi _{13}=\phi _{23}=0$ from the equation \eqref{8}. Then, the system
$\left\{ {U = \frac{{{\phi _u}}}{{\left\| {{\phi _u}} \right\|}},\,\,\,V = \frac{{{\phi _v}}}{{\left\| {{\phi _v}} \right\|}},\,\,W = \frac{{{\phi _w}}}{{\left\| {{\phi _w}} \right\|}}} \right\}$
is an orthonormal frame field. Furthermore, we can write the following equations
\begin{equation}\label{10}
	\begin{array}{l}
		S\left( U \right) = {c_1}\,U + {c_2}\,V + {c_3}\,W\\
		S\left( V \right) = {c_2}\,U + {c_4}\,V + {c_5}\,W\\
		S\left( W \right) = {c_3}U + {c_5}\,V + {c_6}\,W,
	\end{array}
\end{equation}
then, the matrix of the Weingarten map can be calculated as follows:
\[S = \left( {\begin{array}{*{20}{c}}
	{{c_1}}&{{c_2}}&{{c_3}}\\
	{{c_2}}&{{c_4}}&{{c_5}}\\
	{{c_3}}&{{c_5}}&{{c_6}}
	\end{array}} \right).\]
By using the equations \eqref{4}, \eqref{9} and \eqref{10}, the coefficients
$c_{i}\in \mathbb{R},$ $ 1\leq i\leq 6 $ can be calculated as follows:
\begin{equation} \label{11}
	\begin{array}{l}
		{c_1} = \left\langle {S\left( U \right),U} \right\rangle  = \frac{1}{{{{\left\| {{\phi _u}} \right\|}^2}}}\left\langle {\frac{{\partial N}}{{\partial u}},{\phi _u}} \right\rangle ,\\\\
		{c_2} = \left\langle {S\left( U \right),V} \right\rangle  = \frac{1}{{\left\| {{\phi _u}} \right\|}}\frac{1}{{\left\| {{\phi _v}} \right\|}}\left\langle {\frac{{\partial N}}{{\partial u}},{\phi _v}} \right\rangle ,\\\\
		{c_3} = \left\langle {S\left( U \right),W} \right\rangle  = \frac{1}{{\left\| {{\phi _u}} \right\|}}\frac{1}{{\left\| {{\phi _w}} \right\|}}\left\langle {\frac{{\partial N}}{{\partial u}},{\phi _w}} \right\rangle ,\\\\
		{c_4} = \left\langle {S\left( V \right),V} \right\rangle  = \frac{1}{{{{\left\| {{\phi _v}} \right\|}^2}}}\left\langle {\frac{{\partial N}}{{\partial v}},{\phi _v}} \right\rangle ,\\\\
		{c_5} = \left\langle {S\left( V \right),W} \right\rangle  = \frac{1}{{\left\| {{\phi _v}} \right\|}}\frac{1}{{\left\| {{\phi _w}} \right\|}}\left\langle {\frac{{\partial N}}{{\partial v}},{\phi _w}} \right\rangle ,\\\\
		{c_6} = \left\langle {S\left( W \right),W} \right\rangle  = \frac{1}{{{{\left\| {{\phi _w}} \right\|}^2}}}\left\langle {\frac{{\partial N}}{{\partial w}},{\phi _w}} \right\rangle .
	\end{array}
\end{equation}
By using the equation \eqref{5}, we can also write six equations as below:
\begin{equation} \label{12}
	\begin{array}{l}
		\left\langle {\frac{{\partial N}}{{\partial u}},{\phi _u}} \right\rangle  + \left\langle {N,{\phi _{uu}}} \right\rangle  = 0,\\\\
		\left\langle {\frac{{\partial N}}{{\partial u}},{\phi _v}} \right\rangle  + \left\langle {N,{\phi _{uv}}} \right\rangle  = 0,\\\\
		\left\langle {\frac{{\partial N}}{{\partial u}},{\phi _w}} \right\rangle  + \left\langle {N,{\phi _{uw}}} \right\rangle  = 0,\\\\
		\left\langle {\frac{{\partial N}}{{\partial v}},{\phi _v}} \right\rangle  + \left\langle {N,{\phi _{vv}}} \right\rangle  = 0,\\\\
		\left\langle {\frac{{\partial N}}{{\partial v}},{\phi _w}} \right\rangle  + \left\langle {N,{\phi _{vw}}} \right\rangle  = 0,\\\\
		\left\langle {\frac{{\partial N}}{{\partial w}},{\phi _w}} \right\rangle  + \left\langle {N,{\phi _{ww}}} \right\rangle  = 0.
	\end{array}
\end{equation}
Also, by using the equations \eqref{2} and \eqref{8}, we find
\begin{equation} \label{13}
{\left\| {{\phi _u} \otimes {\phi _v} \otimes {\phi _w}} \right\|^2} = \left| {\begin{array}{*{20}{c}}
	\phi _{22}&0&0\\
	0&\phi _{11}&0\\
	0&0&\phi _{33}
	\end{array}} \right| = {\left\| {{\phi _u}} \right\|^2}\,{\left\| {{\phi _v}} \right\|^2}\,{\left\| {{\phi _w}} \right\|^2}.
\end{equation}
Hence we find the coefficients $c_{1},c_{2},c_{3},c_{4},c_{5},c_{6}$ of the Weingarten matrix in the equation \eqref{10} as follows:
\begin{equation} \label{14}
	\begin{array}{l}
		{c_1} =- \frac{1}{{{{\left\| {{\phi _u}} \right\|}^3}}}\frac{1}{{\left\| {{\phi _v}} \right\|}}\frac{1}{{\left\| {{\phi _w}} \right\|}}\det \left( {{\phi _{uu}},{\phi _u},{\phi _v},{\phi _w}} \right),\\\\
		{c_2} =- \frac{1}{{{{\left\| {{\phi _u}} \right\|}^2}}}\frac{1}{{{{\left\| {{\phi _v}} \right\|}^2}}}\frac{1}{{\left\| {{\phi _w}} \right\|}}\det \left( {{\phi _{uv}},{\phi _u},{\phi _v},{\phi _w}} \right),\\\\
		{c_3} =- \frac{1}{{{{\left\| {{\phi _u}} \right\|}^2}}}\frac{1}{{\left\| {{\phi _v}} \right\|}}\frac{1}{{{{\left\| {{\phi _w}} \right\|}^2}}}\det \left( {{\phi _{uw}},{\phi _u},{\phi _v},{\phi _w}} \right),\\\\
		{c_4} =- \frac{1}{{\left\| {{\phi _u}} \right\|}}\frac{1}{{{{\left\| {{\phi _v}} \right\|}^3}}}\frac{1}{{\left\| {{\phi _w}} \right\|}}\det \left( {{\phi _{vv}},{\phi _u},{\phi _v},{\phi _w}} \right),\\\\
		{c_5} =- \frac{1}{{\left\| {{\phi _u}} \right\|}}\frac{1}{{{{\left\| {{\phi _v}} \right\|}^2}}}\frac{1}{{{{\left\| {{\phi _w}} \right\|}^2}}}\det \left( {{\phi _{vw}},{\phi _u},{\phi _v},{\phi _w}} \right),\\\\
		{c_6} =- \frac{1}{{\left\| {{\phi _u}} \right\|}}\frac{1}{{\left\| {{\phi _v}} \right\|}}\frac{1}{{{{\left\| {{\phi _w}} \right\|}^3}}}\det \left( {{\phi _{ww}},{\phi _u},{\phi _v},{\phi _w}} \right).
	\end{array}
\end{equation}
So, by taking into account the equations \eqref{4}, \eqref{13} and \eqref{14} we have the symmetric Weingarten matrix
\begin{equation} \label{15}
	S = \left( {\begin{array}{*{20}{c}}
			{\frac{\varphi_{11}} {\phi _{11}}}&{\frac{\varphi_{12}} {{\sqrt {\phi _{11} \phi _{22}} }}}&{\frac{\varphi_{13}}{{\sqrt {\phi _{11} \phi _{33}} }}}\\\\
			{\frac{\varphi_{12}}{{\sqrt {\phi _{11} \phi _{22}} }}}&{\frac{\varphi_{22}}{\phi _{22}}}&{\frac{\varphi_{23}}{{\sqrt \phi _{22} \phi _{33}}}}\\\\
			{\frac{\varphi_{13}}{{\sqrt { \phi _{11} \phi _{33}} }}}&{\frac{\varphi_{23}}{{\sqrt \phi _{22} \phi _{33}}}}&{\frac{\varphi_{33}}{\phi _{33}}}
	\end{array}} \right).
\end{equation}
where
\begin{equation*}
\begin{array}{l}
\varphi_{11} =- \left\langle {{\phi _{uu}},N} \right\rangle ,\,\,\varphi_{12} =- \left\langle {{\phi _{uv}},N} \right\rangle ,\,\,\varphi_{13} =- \left\langle {{\phi _{uw}},N} \right\rangle,\\
\varphi_{22} = -\left\langle {{\phi _{vv}},N} \right\rangle ,\,\,\varphi_{23} =- \left\langle {{\phi _{vw}},N} \right\rangle ,\,\,\varphi_{33} = -\left\langle {{\phi _{ww}},N} \right\rangle .
\end{array}
\end{equation*}

Finally the following theorem can be given for hypersurface $M^3$ in $E^4$:

\begin{theorem}
	Let $M^3$ be an oriented hypersurface in $E^4.$ Then the Gaussian curvature and the mean curvature of $M^3$ can be given by:
	\begin{equation*}
	    K = \frac{{\varphi_{11} \varphi_{22} \varphi_{33} + 2 \varphi_{12} \varphi_{13} \varphi_{23} - {\varphi_{12}^2} \varphi_{33} - {\varphi_{13}^2} \varphi_{22} - {\varphi_{23}^2} \varphi_{11}}}{{\phi _{11} \phi _{22} \phi _{33}}}
	\end{equation*}
	and
	\begin{equation*}
	    H = \frac{1}{3}\left( {\frac{\varphi_{11}} {\phi _{11}} + \frac{\varphi_{22}} {\phi _{22}} + \frac{\varphi_{33}}{\phi _{33}}} \right),
	\end{equation*}
respectively.
\end{theorem}

\begin{proof}
	By using the equation \eqref{15} and the definitions of the Gaussian curvature $K$ and the mean curvature $H$, the theorem can be easily proved.
\end{proof}

\begin{example}
	Let $M^3$ be an oriented hypersurface with the implicit equation $xy=1$ in $E^4$.  The parametric equation of  $M^3$ can be given by
	\begin{equation*}
	\phi \left( {u,v,w} \right) = \left( {u,\frac{1}{u},v,w} \right).
	\end{equation*}
Then, we obtain ${\phi _u} \otimes {\phi _v} \otimes {\phi _w} = \left( { - \frac{1}{{{u^2}}}, - 1,0,0} \right)$ and the unit normal field \\ $ N = \frac{1}{{\sqrt {1 + {u^4}} }}\left( { - 1, - {u^2},0,0} \right)$. By using the orthonormal basis $ \left\{ {\frac{{{\phi _u}}}{{\left\| {{\phi _u}} \right\|}},\,\,\,\frac{{{\phi _v}}}{{\left\| {{\phi _v}} \right\|}},\,\,\frac{{{\phi _w}}}{{\left\| {{\phi _w}} \right\|}}} \right\}  ,$ \\ we have
\[\begin{array}{l}
S\left( {\frac{{{\phi _u}}}{{\left\| {{\phi _u}} \right\|}}} \right) = \frac{{2{u^3}}}{{{{\left( {1 + {u^4}} \right)}^{3/2}}}}\frac{{{\phi _u}}}{{\left\| {{\phi _u}} \right\|}},\\
S\left( {\frac{{{\phi _v}}}{{\left\| {{\phi _v}} \right\|}}} \right) = 0,\\
S\left( {\frac{{{\phi _w}}}{{\left\| {{\phi _w}} \right\|}}} \right) = 0.
\end{array}\]
So, we find the Weingarten matrix $S$ as:
\[S = \left( {\begin{array}{*{20}{c}}
	{\frac{{2{u^3}}}{{{{\left( {1 + {u^4}} \right)}^{3/2}}}}}&0&0\\
	0&0&0\\
	0&0&0
	\end{array}} \right).\]
\end{example}
\begin{example}
	Let $S^3$ be a hypersphere with the implicit equation ${x^2} + {y^2} + {z^2} + {t^2} = 1$ in $E^4$. The parametric equation of  $S^3$ can be given by
	\begin{equation*}
	\phi \left( {u,v,w} \right) = \left( {\sin u\cos v\sin w,\sin u\sin v\sin w,\cos u\sin w,\cos w} \right). \end{equation*}
Then, $ \{{\phi _u},{\phi _v},{\phi _w}\} $ is an orthogonal system. Also we have the orthonormal basis $ \{U, V,  W\} $ of  $S^3$ such that
\[\begin{array}{l}
U = \frac{{{\phi _u}}}{{\left\| {{\phi _u}} \right\|}} = \left( {\cos u\cos v,\cos u\sin v, - \sin u,0} \right),\,\,\,\\\\
V = \frac{{{\phi _v}}}{{\left\| {{\phi _v}} \right\|}} = \left( { - sinv,\cos v,0,0} \right),\,\\\\
W = \frac{{{\phi _w}}}{{\left\| {{\phi _w}} \right\|}}= \left( {\sin u\cos v\cos w,\sin u\sin v\cos w,\cos u\cos w, - \sin w} \right).
\end{array}\]
Furthermore, the unit normal vector field $N$ can be found:
\begin{equation*}
 N = \left( { - \sin u\cos v\sin w, - \sin u\sin v\sin w, - \cos u\sin w,-\cos w} \right).
\end{equation*}
Then using the equation (15), we obtain $ S = {I_3}. $
\end{example}

\begin{theorem}
	Let $M^3$ be an oriented hypersurface in $E^4$  and let $ \left\{ {{X_P},{Y_P},{Z_P}} \right\} $ be a linearly independent vector system of the tangent space ${T_{{M^3}}}\left( P \right)$. Then, we have
\begin{equation*} \small{
	\begin{array}{l}
	i.\,\,S\left( {{X_P}} \right) \otimes S\left( {{Y_P}} \right) \otimes S\left( {{Z_P}} \right) = K\left( P \right)\left( {{X_P} \otimes {Y_P} \otimes {Z_P}} \right)\\
	ii.\,\,\left( {S\left( {{X_P}} \right) \otimes {Y_P} \otimes {Z_P}} \right) + \left( {{X_P} \otimes S\left( {{Y_P}} \right) \otimes {Z_P}} \right) + \left( {{X_P} \otimes {Y_P} \otimes S\left( {{Z_P}} \right)} \right)= \\
	\,\,\,\,\,\, \,\,\,\,\,\, \,\,\,\,\,\, \,\,\,\,\,\,\,\,\,\,\,\, \,\,\,\,\,\,  3 H\left( P \right)\left( {{X_P} \otimes {Y_P} \otimes {Z_P}} \right),
	\end{array}}
\end{equation*}
	where $K$ and $H$ are the Gaussian curvature and the mean curvature of $M^3$, respectively.
\end{theorem}

\begin{proof}
	By using (i), (ii) parts of the equation \eqref{2} and considering the definitions of the Gaussian curvature $K$ and the mean curvature $H$ the theorem can be easily proved.
\end{proof}
In \cite{nam4}, it is proved that these equations are also provided for closed hypersurfaces.

\begin{theorem}
	Let $M^3$ be an oriented hypersurface in $E^4$  and let $I^q$, $ K $, $ H $ be the $q$-th fundamental forms, the Gaussian curvature and the mean curvature,  respectively. Then we have
	\begin{equation} \label{16}
		{I^4} - 3H\,{I^3} + \frac{{3K}}{h}\,{I^2} - K\,I = 0
	\end{equation}
	where $h$ is the harmonic mean of the non-zero principal curvatures of $M^3.$
\end{theorem}

\begin{proof}
Let ${k_1},{k_2},{k_3}$ be the characteristic values of the Weingarten map $S$ (or the principal curvatures of $M^3$ ). Then  we obtain the characteristic polynomial ${P_S}\left( \lambda  \right)$  of the Weingarten map $S$ of $M^3$ as
\begin{equation*}
    \begin{array}{l}
{P_S}\left( \lambda  \right) = \det \left( {\lambda {I_3} - S} \right){\mkern 1mu} = {\lambda ^3} - \left( {{k_1} + {k_2} + {k_3}} \right){\lambda ^2}{\mkern 1mu}   + \left( {{k_1}{k_2} + {k_1}{k_3} + {k_2}{k_3}} \right)\lambda  - \left( {{k_1}{k_2}{k_3}} \right).
\end{array}
\end{equation*}

By using the Cayley-Hamilton theorem, we obtain
\[{S^3} - \left( {{k_1} + {k_2} + {k_3}} \right){S^2} + \left( {{k_1}{k_2} + {k_1}{k_3} + {k_2}{k_3}} \right)S - \left( {{k_1}{k_2}{k_3}} \right){I_3} = 0.\]
By using the definitions of the $q-th$ fundamental forms, the Gaussian curvature, the mean curvature and the harmonic mean
\begin{equation*}
h = \frac{3}{{\frac{1}{{{k_1}}} + \frac{1}{{{k_2}}} + \frac{1}{{{k_3}}}}}
\end{equation*}
of the principal curvature ${k_1},{k_2},{k_3}$,\\
we obtain the equation \eqref{16}.
\end{proof}


\section{Dupin indicatrix of the hypersurface in $E^4$}
\addcontentsline{toc}{section}{ } 

Let $X,Y,Z$ be three principal vectors according to the principal curvatures ${k_1},{k_2},{k_3}$ of $M^3$. If we consider the orthonormal basis $\{X,Y,Z\}$ of $M^3$ then for any tangent vector $W_{P}\in T_{M^3}(P),$ we can write
${W_P} = x\,{X_P} + y\,{Y_P} + z\,{Z_P},\,\, $ where $x,y,z \in \mathbb{R}, $
and
\[\begin{array}{l}
S\left( {{W_P}} \right) = x\,S\left( {{X_P}} \right) + y\,S\left( {{Y_P}} \right) + z\,S\left( {{Z_P}} \right)\\
\,\,\,\,\,\,\,\,\,\,\,\,\,\,\,\,\,\,\,\,\,\,\,\,\,\, = x\, {k_1}{X_P} + y\,{k_2}{Y_P} + z\,{k_3}{Z_P}
\end{array}\]
Here, the Dupin indicatrix $\mathbb{D}$ of $M^3$ can be defined by

\[\mathbb{D}=\left\{ \begin{array}{l}
{W_P} = \left( {x,y,z} \right) \in {T_{{M^3}}}(P)|
\left\langle {S\left( {{W_P}} \right),{W_P}} \right\rangle  = {k_1}{x^2} + {k_2}{y^2} + {k_3}{z^2} =  \pm 1
\end{array} \right\}.\]
In another words, the Dupin indicatrix corresponds to a\\ hypercylinder which has the equation
\begin{equation*}
    {k_1}{x^2} + {k_2}{y^2} + {k_3}{z^2} =  \pm 1.
\end{equation*}
Now, we will examine the Dupin indicatrix according to the Gaussian curvature $K:$\\
1)\,\, Let $K\left( P \right) > 0. $
\begin{itemize}
\item If ${k_1},\,{k_2},\,{k_3} > 0$ then for equation of the Dupin indicatrix, we can write
${k_1}{x^2} + {k_2}{y^2} + {k_3}{z^2} =  \pm 1.$
Hence, the Dupin indicatrix is the ellipsoidal class and this equation is called ellipsoidal cylinder in $E^4.$ In this condition,\,\,\, $P\in M^3$ is called an ellipsoidal point.
\item If ${k_1}> 0,\,\, {k_2},\,{k_3} < 0$ or ${k_2}> 0,\,\, {k_1},\,{k_3} < 0$ or ${k_3}> 0$ ${k_1},\,{k_2} < 0  $ then for equation of the Dupin indicatrix, we can write
${k_1}{x^2} - {k_2}{y^2} - {k_3}{z^2} =  \pm 1.$ Hence, the Dupin indicatrix is the hyperboloidical class and this equation is called hyperboloidical cylinder one or two sheets in $E^4$. In this condition, $P\in M^3$ is called a hyperboloidical point.
\end{itemize}
2)\,\, Let $K\left( P \right) < 0.$
\begin{itemize}
\item If only one of ${k_i}$'s,\,  ${i=1,2,3}$ is negative,  then for the equation of the Dupin indicatrix, we can write\\
\[\left\{ \begin{array}{l}
{k_1}{x^2} + {k_2}{y^2} - {k_3}{z^2} =  \pm 1, \\
{k_1}{x^2} - {k_2}{y^2} + {k_3}{z^2} =  \pm 1, \\
- {k_1}{x^2} + {k_2}{y^2} + {k_3}{z^2} =  \pm 1.
\end{array} \right.\]
The above equations are called one or two sheeted hyperboloidical cylinder in $E^4.$ Then $ P\in M^3$ is called a hyperboloidical point.

\item If ${k_1},{k_2},{k_3}< 0$ then the Dupin indicatrix is the ellipsoidal class and this equation is called ellipsoidal cylinder in $E^4.$ So $P\in M^3$ is called a ellipsoidal point.
\end{itemize}
3) Let $K\left( P \right) = 0$.
\begin{itemize}
\item If ${k_1}=0$  or ${k_2}=0\,\ or \,\ {k_3}=0 $, then for the equation of the Dupin indicatrix for each case, we get
\begin{itemize}
	\item[i] If ${k_1} = 0,\,{k_2},\,{k_3}$ are the same or different signs then ${k_2}{y^2} + {k_3}{z^2} =  \pm 1$.
	\item[ii]  If ${k_2} = 0,\,{k_1}, {k_3}$ are the same or different signs then ${k_1}{x^2} + {k_3}{z^2} =  \pm 1$.
	\item[iii] If ${k_3} = 0,\,{k_1},\,{k_2}$ are the same or different signs then ${k_1}{x^2} + {k_2}{y^2} =  \pm 1.$
\end{itemize}
These equations are called elliptic cylinder or hyperbolic cylinder in $E^4.$ In this condition, $ P\in M^3$ is called an elliptic cylinder or hyperbolic cylinder point.

\item If ${k_1} = {k_2} = {k_3} = 0$ then the point $P \in M^3$ is a flat point.

\item  If any two of ${k_i}$'s,\,  ${i=1,2,3}$ are zero and other positive or negative then ${k_3}{z^2} =  \pm 1$ or ${k_2}{y^2} =  \pm 1$ or ${k_1}{x^2} =  \pm 1.$
	\end{itemize}



\end{document}